\documentclass[fleqn,reqno,10pt,a4paper,final]{amsart}

\usepackage[a4paper,left=35mm,right=35mm,top=30mm,bottom=30mm,marginpar=25mm]{geometry} 
\usepackage{amsmath}
\usepackage{amssymb}
\usepackage{amsthm}
\usepackage{amscd}
\usepackage[ansinew]{inputenc}
\usepackage{bbm}
\usepackage{color}
\usepackage[final]{graphicx}
\usepackage{hyperref}

\numberwithin{equation}{section}

\newtheoremstyle{thmlemcorr}{10pt}{10pt}{\itshape}{}{\bfseries}{.}{10pt}{{\thmname{#1}\thmnumber{ #2}\thmnote{ (#3)}}}
\newtheoremstyle{thmlemcorr*}{10pt}{10pt}{\itshape}{}{\bfseries}{.}\newline{{\thmname{#1}\thmnumber{ #2}\thmnote{ (#3)}}}
\newtheoremstyle{defi}{10pt}{10pt}{\itshape}{}{\bfseries}{.}{10pt}{{\thmname{#1}\thmnumber{ #2}\thmnote{ (#3)}}}
\newtheoremstyle{remexample}{10pt}{10pt}{}{}{\bfseries}{.}{10pt}{{\thmname{#1}\thmnumber{ #2}\thmnote{ (#3)}}}
\newtheoremstyle{ass}{10pt}{10pt}{}{}{\bfseries}{.}{10pt}{{\thmname{#1}\thmnumber{ A#2}\thmnote{ (#3)}}}

\theoremstyle{thmlemcorr}
\newtheorem{theorem}{Theorem}
\newtheorem{lemma}[theorem]{Lemma}

\newtheorem{proposition}[theorem]{Proposition}

\theoremstyle{thmlemcorr*}
\newtheorem{theorem*}{Theorem}
\newtheorem{lemma*}[theorem]{Lemma}
\newtheorem{corollary*}[theorem]{Corollary}
\newtheorem{proposition*}[theorem]{Proposition}
\newtheorem{problem*}[theorem]{Problem}
\newtheorem{conjecture*}[theorem]{Conjecture}

\theoremstyle{defi}
\newtheorem{definition}[theorem]{Definition}

\theoremstyle{remexample}

\theoremstyle{ass}

\newcommand{\Crm}{\mathrm{C}}

\newcommand{\Lrm}{\mathrm{L}}

\newcommand{\Wrm}{\mathrm{W}}

\newcommand{\Hcal}{\mathcal{H}}

\newcommand{\Lcal}{\mathcal{L}}

\newcommand{\Qcal}{\mathcal{Q}}
\newcommand{\Rcal}{\mathcal{R}}

\newcommand{\Tcal}{\mathcal{T}}

\DeclareMathOperator{\diam}{diam}

\DeclareMathOperator{\dist}{dist}
\DeclareMathOperator{\rank}{rank}

\newcommand{\setBBB}[2]{\Biggl\{\, #1 \ \ \textup{\textbf{:}}\ \ #2 \,\Biggr\}}
\newcommand{\norm}[1]{\|#1\|}

\newcommand{\abs}[1]{|#1|}

\newcommand{\absb}[1]{\bigl|#1\bigr|}
\newcommand{\absB}[1]{\Bigl|#1\Bigr|}

\newcommand{\cl}[1]{\overline{#1}}
\newcommand{\di}{\mathrm{d}}
\newcommand{\dd}{\;\mathrm{d}}

\newcommand{\N}{\mathbb{N}}
\newcommand{\R}{\mathbb{R}}

\newcommand{\ONE}{\mathbbm{1}}

\newcommand{\toweakstar}{\overset{*}\rightharpoondown}

\newcommand{\todown}{\downarrow}

\newcommand{\conv}{\star}

\newcommand{\sbullet}{\begin{picture}(1,1)(-0.5,-2.5)\circle*{2}\end{picture}}
\newcommand{\frarg}{\,\sbullet\,}
\newcommand{\BV}{\mathrm{BV}}

\newcommand{\eps}{\epsilon}
\newcommand{\lrangle}[1]{\langle #1 \rangle}

\newcommand{\term}[1]{\textbf{#1}}

\newcommand{\proofstep}[1]{\textit{#1}}

\def\Xint#1{\mathchoice 
{\XXint\displaystyle\textstyle{#1}}%
{\XXint\textstyle\scriptstyle{#1}}%
{\XXint\scriptstyle\scriptscriptstyle{#1}}%
{\XXint\scriptscriptstyle\scriptscriptstyle{#1}}%
\!\int} 
\def\XXint#1#2#3{{\setbox0=\hbox{$#1{#2#3}{\int}$} 
\vcenter{\hbox{$#2#3$}}\kern-.5\wd0}} 
\def\dashint{\,\Xint-}

\newcommand{\restrict}{\begin{picture}(10,8)\put(2,0){\line(0,1){7}}\put(1.8,0){\line(1,0){7}}\end{picture}}

\renewcommand{\epsilon}{\varepsilon}
\renewcommand{\phi}{\varphi}

\title{Piecewise affine approximations for functions of bounded variation}

\author{Jan Kristensen}
\address{Mathematical Institute, University of Oxford, 24--29 St Giles', Oxford OX1 3LB, United Kingdom.}
\email{kristens@maths.ox.ac.uk}

\author{Filip Rindler}
\address{Cambridge Centre for Analysis, University of Cambridge, Centre for Mathematical Sciences, Wilberforce Road, Cambridge CB3 0WA, United Kingdom.}
\email{F.Rindler@maths.cam.ac.uk}

\begin{document}

\begin{abstract}
BV functions cannot be approximated well by piecewise constant functions, but this work will show that a good approximation is still possible 
with (countably) piecewise affine functions. In particular, this approximation is area-strictly close to the original function and the $\mathrm{L}^1$-difference 
between the traces of the original and approximating functions on a substantial part of the mesh can be made arbitrarily small. Necessarily, the mesh needs to be adapted to the 
singularities of the BV function to be approximated, and consequently, the proof is based on a blow-up argument together with explicit constructions of the mesh. 
In the case of $\Wrm^{1,1}$-Sobolev functions we establish an optimal $\mathrm{W}^{1,1}$-error estimate for approximation by piecewise affine functions on uniform regular
triangulations. 
The piecewise affine functions are standard quasi-interpolants obtained by mollification and Lagrange interpolation on the nodes of triangulations, and the main new contribution 
here compared to for instance Cl\'{e}ment (RAIRO Analyse Num\'{e}rique 9 (1975), no.~R-2, 77--84) and Verf\"{u}rth (M2AN Math.~Model.~Numer.~Anal. 33 (1999), no.~4, 695--713) 
is that our error estimates are in the $\mathrm{W}^{1,1}$-norm rather than merely the $\mathrm{L}^1$-norm.  
\vspace{4pt}

\noindent\textsc{MSC (2010): 41A10 (primary); 26A45, 26B30} 

\noindent\textsc{Keywords:} Approximation, quasi-interpolant, BV, bounded variation, $\mathrm{W}^{1,1}$, piecewise affine, finite element.

\vspace{4pt}

\noindent\textsc{Date:} \today{} (version 2.0).
\end{abstract}

\hypersetup{
  pdfauthor = {Jan Kristensen (University of Oxford), Filip Rindler (University of Warwick)},
  pdftitle = {Piecewise affine approximations for functions of bounded variation},
  pdfsubject = {MSC (2010): 41A10 (primary); 26A45, 26B30},
  pdfkeywords = {Approximation, Lagrange interpolation, BV, bounded variation, W^{1,1}, piecewise affine}
}

\maketitle


\section{Introduction}

Functions of bounded variation have important applications in many branches of mathematical physics, 
among them optimization~\cite{AtBuMi06VASB}, free-discontinuity problems~\cite{AmFuPa00FBVF} 
(this and the previous reference also contain good introductions to the theory of BV functions), and hyperbolic systems of conservation 
laws~\cite{Dafe10HCLC}. However, apart from applications to image segmentation and related models (see \cite{BoCh} and the references 
quoted there), the theory of numerical approximations of such functions is not very well developed and indeed, because 
of the nature of singularities of BV functions, a more thorough analysis is required. In this work we consider the basic question of 
whether a BV function can be approximated well by (countably) piecewise affine functions. Before we start our investigation, though, we remark 
that any good approximation by piecewise-constant BV functions must fail, this will be shown in Proposition~\ref{prop:no_const_approx} below.

For $d > 1$ an integer, let $\Omega$ be a bounded Lipschitz domain in $\R^d$. The density of piecewise affine functions in 
$\Wrm^{1,p}(\Omega;\R^m)$ for exponents $1 \leq p < \infty$ is fundamental for finite element methods, and for example also useful in 
proving lower semicontinuity of variational functionals, see for example~\cite{GofSer64SFMV,Gius83MSFB,Daco08DMCV}. Such an approximation result 
is also easily seen to be true in the space BV of functions of bounded variation, if we switch from norm convergence to the so-called 
(area-)strict convergence (see below).

For some applications, however, one needs a better approximation result, in which also the trace differences over the boundaries of the mesh 
are controlled. This need becomes immediately obvious for example in the analysis of energy functionals in free-discontinuity problems, which 
typically involve a term measuring jumps over the discontinuity surface (see for example~\cite{BraCos}), and where, to produce good 
approximants, one needs to control also the approximation of jumps.

The purpose of this note is to prove the following approximation theorem, related to a weaker result in~\cite{KriRin12CGGYErr} 
and the numerical and analytical studies~\cite{Bart11TVMF-pre,Wald98MTF,QdGr08SASB,AmaCic05NARB,BDKPW07GWPB,Wojt03PNAS}.

\begin{theorem}\label{thm:approx}
Let $\Omega \subset \R^d$ be a bounded Lipschitz domain, $\lambda$ a finite Borel measure on $\Omega$ and $u \in \BV(\Omega;\R^m)$. 
Then, for every $\eps > 0$ there exist a countable family $\Rcal$ of (rotated) rectangles and simplices $R \subset \Omega$ and 
$v \in \Wrm^{1,1}(\Omega;\R^m)$ with the following properties:
\begin{itemize}
\item[(i)] The sets in $\Rcal$ are non-overlapping and $(\Lcal^d + \abs{Du})(\Omega \setminus \bigcup \Rcal )=0$.
\item[(ii)] For each $R \in \Rcal$ the restriction $v|_{R}$ is affine.
\item[(iii)] $\norm{u-v}_{\Lrm^1(\Omega;\R^m)} + \abs{\lrangle{Du}(\Omega) - \lrangle{Dv}(\Omega)} < \eps$.
\item[(iv)] $\Rcal = \Rcal_{g} \cup \Rcal_b$ with $\lambda \left( \bigcup \Rcal_{b} \right) < \varepsilon$ and
$\displaystyle\sum_{R \in \Rcal_g} \int_{\partial R} \abs{u-v} \dd \Hcal^{d-1} < \eps$.
\item[(v)] $v|_{\partial \Omega} = u|_{\partial \Omega}$.
\end{itemize}
\end{theorem}
We refer to the family $\Rcal$ as a mesh for $\Omega$. Observe that $\Rcal$ need not be locally finite in $\Omega$ and therefore the
$\Wrm^{1,1}$-function $v$ in the theorem, called an $\Rcal$-\term{piecewise affine function}, need not be continuous. The property (v) should
be understood in the sense of trace in BV (see \cite{AmFuPa00FBVF}).

Let us also briefly comment on the notation for measures. If $\mu = \frac{\di \mu}{\di \Lcal^d} \, \Lcal^d + \mu^s$ is the 
Lebesgue--Radon--Nikod\'{y}m decomposition of the measure $\mu$, we define the following measure (related to the area functional)
\[
  \lrangle{\mu}(A) := \abs{( \mu , \Lcal^d )}(A) = \int_A \sqrt{1 + \absB{\frac{\di \mu}{\di \Lcal^d}}^2 }
    \dd x + \abs{\mu^s}(A)
\]
for every Borel set $A \subset \Omega$, where $\abs{( \mu , \Lcal^d )}$ and $\abs{ \mu^s}$ denote the total variation measures of the vector measures
$( \mu , \Lcal^d )$ and $\mu^s$, respectively. We always define the total variation measures with respect to the natural euclidean norms 
(see \cite{AmFuPa00FBVF}).

Strong convergence in BV is not very useful here and instead we shall mainly be concerned with three weaker notions of convergence in BV:
A sequence $(u_j) \subset \BV(\Omega;\R^m)$ is said to converge \term{weakly}$\mbox{}^{\ast}$ to $u \in \BV(\Omega;\R^m)$ if $\norm{u_j-u}_{\Lrm^1} \to 0$
and $Du_{j} \toweakstar Du$ in $\Crm_{0}(\Omega , \R^{m \times d})^{\ast}$, it is said to converge \term{strictly} if
$\norm{u_j-u}_{\Lrm^1} \to 0$ and $\abs{Du_j}(\Omega) \to \abs{Du}(\Omega)$ as $j \to \infty$; and it is said to converge \term{$\lrangle{\frarg}$-strictly} 
(or \term{area-strictly}) if $\norm{u_j-u}_{\Lrm^1} \to 0$ and $\lrangle{Du_j}(\Omega) \to \lrangle{Du}(\Omega)$. The latter mode of convergence is the 
natural replacement for the strong convergence in $\BV(\Omega;\R^m)$ since one can show that smooth functions are $\lrangle{\frarg}$-strictly but not strongly 
dense in $\BV(\Omega;\R^m)$. In fact, we even have the following stronger result on $\lrangle{\frarg}$-strict density of $\Crm^\infty$ smooth functions under
trace constraints (see Lemma~1 of~\cite{KriRin10CGGY} for a slightly more general statement or Lemma~B.1 in~\cite{Bild03CVP}):

\begin{lemma}\label{lem:smooth_approx}
Let $\Omega \subset \R^d$ be a bounded Lipschitz domain and let 
$u \in \BV(\Omega;\R^m)$. There exists a sequence $(v_j) \subset (\Wrm^{1,1} \cap \Crm^{\infty})(\Omega;\R^m)$ such that
\[
  v_j \to u \quad\text{$\lrangle{\frarg}$-strictly}
  \, \text{ and } \,
  v_j = u \, \text{ in the sense of trace on $\partial \Omega$ for all $j \in \N$.}
\]
If $u \in \Wrm^{1,1}(\Omega ; \R^m )$, then we can in addition arrange that $v_j \to u$ strongly in $\Wrm^{1,1}(\Omega ; \R^m )$.
\end{lemma}

See also \cite{TS} for an interesting generalization.
Many reasonable integral functionals with linear growth are continuous with respect to the $\lrangle{\frarg}$-strict convergence by Reshetnyak's 
continuity theorem~\cite{Resh68WCCA} and extensions, see for example Theorem~4 and the appendix of~\cite{KriRin10RSIF}; lower semicontinuity in this 
situation is discussed in~\cite{Dell91LSCF}. 

The main result of this paper establishes density of (countably) piecewise affine functions in BV in the area-strict sense and such that additionally 
conditions (iv), (v) are satisfied. We remark that it is easy to satisfy conditions (i)--(iii) of 
Theorem~\ref{thm:approx} by simply mollifying $u$, which gives a strictly close smooth function (by Lemma~\ref{lem:smooth_approx}) 
and then strongly approximating this smooth function by a piecewise affine function. Unfortunately, on this route one loses control of the sum of 
integrals in (iv), and it is precisely this condition, which is important for instance in applications to free-discontinuity problems.

In Theorem \ref{thm:Lagrange_W11} below we establish substantially stronger approximation results for $\Wrm^{1,1}$-Sobolev functions. This is done by showing
that suitable quasi-interpolants on any given regular and uniform triangulation yield an approximation with error bounds that only depend on the $\Lrm^1$-modulus of 
continuity of the weak gradient of the approximated $\Wrm^{1,1}$-function and regularity/uniformity constants of the triangulation. We refer to Section 2 for
notation and terminology. Hence, by virtue of the Kolmogorov-Riesz characterization of compactness in $\Lrm^1$, we have in particular a uniform 
rate of approximation by quasi-interpolants on all norm-compact subsets of $\Wrm^{1,1}$. Of course, similar results remain true in $\Wrm^{1,p}$  for $p > 1$, but we shall
refrain from stating these results explicitly here and instead focus on the case $p=1$.

\begin{theorem} \label{thm:Lagrange_W11}
Assume that $\Tcal$ is a regular triangulation of $\R^d$ which is also uniform in the sense that there exists a 
constant $\gamma \in (0,1)$ such that
\[
  \gamma k \leq \diam \tau \leq k  \qquad\text{for all $\tau \in \Tcal$,}
\]
where
\[
  k := \sup_{\tau \in \Tcal} \diam \tau.
\]
Then, there exists a constant $C$, only depending on the dimensions, $\gamma$, and the parameter of regularity of the triangulation $\Tcal$, such that
for any bounded Lipschitz domain $\Omega$ in $\R^d$ and any Sobolev function $u \in \Wrm^{1,1}(\Omega;\R^m)$ there exists a $\Tcal$-piecewise
affine function $a \in \Wrm^{1,\infty}(\Omega , \R^m )$ with
\[
  \int_{\Omega} \bigl( \abs{u-a}+\abs{\nabla u-\nabla a} \bigr) \dd x \leq C \omega(3k),
\]
where $\omega$ is a suitable $\Lrm^1$-modulus of continuity for $\nabla u$; a modulus of continuity is suitable if it is a $\Lrm^1$-modulus of continuity for $\nabla U$ where $U$ is any $\Wrm^{1,1}$-extension of
$u$ from $\Omega$ to $\R^d$, see~\eqref{eq:mod_cont} for the precise definition.
\end{theorem}

We end this introduction by returning to the issue of approximating with pure jump functions. As mentioned above, if in Theorem~\ref{thm:approx} we try
to approximate with pure jump functions, that is, (countably) piecewise constant BV functions, instead of piecewise affine functions, then the mode
of approximation cannot be area-strict. This also gives another reason why condition (iv) in 
Theorem~\ref{thm:approx} is of interest:

\begin{proposition} \label{prop:no_const_approx}
Let $m \in \N$ be a natural number and assume that $u \in \Crm^{\infty}_{c}(B(0,1), \R^{m})$ is not identically zero. Then there can be no sequence 
of piecewise constant functions $u_{j} \in \BV (B(0,1), \R^{m})$ such that $u_{j} \to u$ $\lrangle{\frarg}$-strictly in $\BV$.
\end{proposition}

\begin{proof}
Let $F\colon \R^{m \times d} \to \R$ be any continuous integrand satisfying $F \equiv 0$ on $\nabla u(B(0,1))$, which is a bounded set by assumption, and 
such that $F(z) = |z|$ for large values of $|z|$. Then clearly its recession function equals $|z|$:
\[
F^\infty (z) := \lim_{\stackrel{t \to \infty}{z^{\prime} \to z}} \frac{F(tz^{\prime})}{t} = |z|.
\]
Now if $u_j$ were piecewise constant $\BV$ functions such that $u_{j} \to u$ $\lrangle{\frarg}$-strictly in $\BV$, then by Reshetnyak's continuity 
theorem~\cite{Resh68WCCA} (also see~\cite{Spec11SPRR} and the appendix of~\cite{KriRin10RSIF}), we would have 
\[
  \abs{Du_j} \toweakstar \abs{Du}
\]
and, since $Du_{j} = D^{s}u_{j}$,
\[
  \abs{Du_j} = F^\infty \biggl( \frac{\di D^s u_j}{\di \abs{D^s u_j}} \biggr) \, \abs{D^s u_j}
  \toweakstar F \biggl( \frac{\di Du}{\di \Lcal^d} \biggr) \, \Lcal^d = 0,
\]
where both convergences are weak* in $\Crm_0 (B(0,1))^*$. But this is clearly impossible because $Du \neq 0$.
\end{proof}

For another example for the usual strict convergence in the vectorial case, consider the continuous and $1$-homogeneous integrand
\[
  F \colon \R^{2 \times 2} \to \R,  \qquad
  F(z) := \abs{\det z}^{1/2}.
\]
Then, for any $u_j \to u$ strictly in $\BV(\R^2;\R^2)$, Reshetnyak's continuity theorem  implies
\[
  F \biggl( \frac{\di Du_j}{\di \abs{Du_j}} \biggr) \, \abs{Du_j} \toweakstar F \biggl( \frac{\di Du}{\di \abs{Du}} \biggr) \, \abs{Du} =: F(Du).
\]
However, if the $u_j$'s are pure jump functions, then $\rank \frac{\di Du_j}{\di \abs{Du_j}} \leq 1$ $\abs{Du_j}$-almost everywhere, but $F$ is zero 
on the rank-one cone, and hence $F(Du) = 0$ as a measure would follow. Thus it is not possible to approximate in the strict sense any BV function
$u$ with $F(Du) \neq 0$ by piecewise constant $\BV$ functions.

This observation should be contrasted with the case of real-valued BV-functions (corresponding to $m=1$) and the usual strict convergence, where 
a discretisation of the coarea formula allows strict approximation by pure jump functions, see~\cite{BraCos} and \cite[Theorem~3]{Wojt03PNAS}.

\section*{Acknowledgements}
The authors wish to thank Nicola Fusco and Endre S\"{u}li for discussions related to the subject of this paper.

\section{Approximation on a triangulation and proof of Theorem \ref{thm:Lagrange_W11}}

We use standard notation and terminology for simplices and triangulations, and the reader can find all necessary background
in the monograph \cite{moise}. However, it is convenient to recall the key concepts and definitions.
Let $\tau$ be a \term{$d$-simplex} in $\R^d$, that is, $\tau$ is the convex hull of $d+1$ affinely independent points $v_{0}, \, \dots \, , \, v_d$ of $\R^d$, the \term{vertices} of $\tau$. A face of $\tau$ is any convex combination of a (non-empty) subset of its vertices, and as such it is a lower dimensional simplex.

Each point $x \in \tau$ admits a unique \term{barycentric representation} as
a convex combination of the vertices of $\tau$:
\[
x = \sum_{j=0}^{d} \lambda_{j} v_{j},
\]
where $\lambda_{j} = \lambda_{j}(x) \in [0,1]$, $j = 1,\ldots,d$, and $\sum_{j=0}^{d} \lambda_{j} \equiv 1$. The functions $\tau \ni x \mapsto \lambda_{j}(x) \in [0,1]$
are easily seen to be restrictions of affine functions (again denoted) $\lambda_{j}\colon \R^d \to \R$ satisfying $\lambda_{j}(\tau ) =[0,1]$ and
$\sum_{j=0}^{d} \lambda_{j} \equiv 1$ on $\R^d$. If $f\colon \tau \to \R^m$ is a function, then the Lagrange
interpolation of $f$ on $\tau$ is the (unique) affine function $a\colon \R^d \to \R^m$ that agrees with $f$ at the vertices of $\tau$. In terms of the $\lambda_j$'s, 
it can be written as
\[
  a(x) = \sum_{j=0}^\infty \lambda_j(x) f(v_j).
\]
Let us also record that the standard $d$-simplex 
\[
\sigma = \sigma^d := \setBBB{ x \in \R^d }{ x_{j} \geq 0 \mbox{ for all } j, \, \sum_{j=1}^{d}x_{j} \leq 1 }
\] 
can be used as reference in the sense that given any $d$-simplex $\tau$ we can find a vector $c \in \R^d$ and a linear automorphism $M \in \mathrm{GL}(d)$ 
such that 
\begin{equation}\label{refsim}
\tau = c+M\sigma .
\end{equation}
This remark is useful when we define regularity of a triangulation below.

On a $d$-simplex we have the following elementary bound, where we note that the right-hand side is the local Riesz potential of the second derivative evaluated
at the vertices of the simplex. This is the natural counterpart in the context of Lagrange interpolation of the standard potential estimates obtained by expressing 
functions in terms of their derivatives and the Newtonian potential.

\begin{lemma} \label{lem:Lagrange_est}
Let $\tau = c+M\sigma$, where $c \in \R^d$ and $M \in \mathrm{GL}(d)$ and assume that $1/\alpha \leq \det M \leq \alpha$ for some constant $\alpha >1$.
There exists a constant $C=C(d,m,\alpha )$, only depending on the dimensions and the regularity constant $\alpha$, with the following property: Assume that 
$f\colon \tau \to \R^m$ is a $\Crm^2$ function with Lagrange interpolant $a$ on the $d$-dimensional simplex $\tau$.
Then,
\[
\int_\tau \abs{\nabla f -\nabla a} \dd x \leq C (\diam \tau)^d \sum_{j=0}^{d} \int_\tau \frac{\abs{\nabla^{2}f(x)}}{\abs{x-v_{j}}^{d-1}} \dd x.
\]
\end{lemma}

\begin{proof}
We start from the formula
\[
\sum_{j=0}^{d} (x-v_j) \otimes \nabla \lambda_{j}(x) = -I  \quad \in \R^{d \times d},
\]
which is a consequence of $\sum_j \nabla \lambda_j \equiv 0$. Thus,
\begin{align*}
  \nabla f(x)-\nabla a(x) &= - \nabla f(x) \biggr( \sum_{j=0}^{d} (x-v_j) \otimes \nabla \lambda_{j} \biggl)  + \sum_{j=0}^{d} \bigl( f(x) - f(v_j) \bigr) \otimes \nabla \lambda_{j}\\
  &= \sum_{j=0}^{d} \biggl( \int_{0}^{1} \bigl( \nabla f(v_{j}+t(x-v_{j})) -\nabla f(x) \bigr) \cdot (x-v_{j}) \dd t \biggr) \otimes \nabla \lambda_{j} \\
  &= \sum_{j=0}^{d} \biggl( \int_{0}^{1} \int_1^t (x-v_{j})^T \, \nabla^2 f(v_{j}+s(x-v_{j})) \, (x-v_{j}) \dd s \dd t \biggr) \otimes \nabla \lambda_{j}.
\end{align*}
Consequently, we find
\begin{align}
&\int_\tau \abs{\nabla f(x)-\nabla a(x)} \dd x   \notag\\
&\qquad \leq \sum_{j=0}^{d} \int_\tau \abs{\nabla \lambda_{j}} \int_{0}^{1} \abs{\nabla^{2}f(v_{j}+s(x-v_{j}))} \cdot \abs{x-v_{j}}^2 \dd s \dd x   \notag\\
&\qquad = \sum_{j=0}^{d} \int_{0}^{1} \int_{0}^{1} \int_{\{\lambda_{j}=h\} \cap \tau} \abs{\nabla^{2}f(v_{j}+s(x-v_{j}))} \cdot \abs{x-v_{j}}^2 \dd \Hcal^{d-1}(x) \dd s \dd h,  \label{eq:simplex_3int}
\end{align}
where we used (an elementary version of) the coarea formula for the last equality. Now, for each fixed $h \in (0,1)$ we change variables via 
$y = \Phi(x,s) := v_{j}+s(x-v_{j})$, where $(x,s) \in (\tau \cap \lambda_{j}^{-1}\{ h \}) \times (0,1)$ and $y \in \tau_h := \tau \cap \{\lambda_j \geq h\}$, for which we can estimate
\[
  \di \Hcal^{d-1}(x) \dd s = \frac{\di y}{\abs{\det D\Phi}}
    \sim \frac{\di y}{s^{d-1} h \cdot \diam \tau} \sim \frac{(\diam \tau)^{d-2}}{\abs{y-v_j}^{d-1} \cdot h} \dd y,
\]
since $(\diam \tau) s \sim \abs{y-v_j}$, and where the implied constants depend on $\alpha$ and the dimensions. Furthermore, $\abs{x-v_j} \sim (\diam \tau) h$ and hence
\begin{align*}
  &\int_{0}^{1} \int_{\{\lambda_{j}=h\}\cap \tau} \abs{\nabla^{2}f(v_{j}+s(x-v_{j}))} \cdot \abs{x-v_{j}}^2 \dd \Hcal^{d-1}(x) \dd s \\
  &\qquad \leq C (\diam \tau)^2 \int_{0}^{1} \int_{\{\lambda_{j}=h\}\cap \tau} \abs{\nabla^{2}f(v_{j}+s(x-v_{j}))} \cdot h^2 \dd \Hcal^{d-1}(x) \dd s \\
  &\qquad \leq C (\diam \tau)^d \int_{\tau_h} \frac{\abs{\nabla^{2}f(y)}}{\abs{y-v_j}^{d-1}} \cdot h \dd y.
\end{align*}
Plugging this into~\eqref{eq:simplex_3int}, we conclude
\begin{align*}
  \int_\tau \abs{\nabla f(x)-\nabla a(x)} \dd x
  &\leq C (\diam \tau)^d \sum_{j=0}^{d} \int_{0}^{1} \int_{\tau_h} \frac{\abs{\nabla^{2}f(y)}}{\abs{y-v_j}^{d-1}} \cdot h \dd y \dd h \\
  &\leq C (\diam \tau)^d \sum_{j=0}^{d} \int_{\tau} \frac{\abs{\nabla^{2}f(y)}}{\abs{y-v_j}^{d-1}} \dd y.
\end{align*}
This completes the proof.
\end{proof}


Next we consider triangulations (where, unless stated otherwise, triangulations could possibly be infinite). We start with some standard notation and terminology:

\begin{definition}\label{regular}
A \term{triangulation} of an open subset $\Omega \subset \R^d$ is a family $\Tcal$ of $d$-simplices $\tau \subset \Omega$ satisfying the following
three conditions:
\begin{itemize}
\item[(i)] $\cl{\Omega} = \cl{\bigcup \Tcal}$;
\item[(ii)] if $\tau^{\prime}$, $\tau^{\prime \prime} \in \Tcal$ and $\tau^{\prime} \cap \tau^{\prime \prime} \neq \emptyset$, then $\tau^{\prime} \cap \tau^{\prime \prime}$ is a common face;
\item[(iii)] every $\tau \in \Tcal$ lies in an open set $V$ which intersects only a finite number of simplices from $\Tcal$.
\end{itemize}
The triangulation $\Tcal$ is called \term{regular} if there exists a constant $\alpha \geq 1$ such that whenever $\tau \in \Tcal$ is represented as in (\ref{refsim}) then
\[
  \alpha^{-1} \leq \abs{\det M} \leq \alpha .
\]
A function $f \colon \Omega \to \R^m$ is said to be piecewise affine with respect to $\Tcal$ (or $\Tcal$-piecewise affine) if for each $\tau \in \Tcal$
the restriction $f|_{\tau}$ is an affine function (that is, equals the restriction $a|_{\tau}$ of an affine function $a \colon \R^d \to \R^m$).
\end{definition}
Note that triangulations $\Tcal$ must be locally finite in $\Omega$, but that they can become infinite towards the boundary of $\Omega$. Therefore a $\Tcal$-piecewise affine 
$\Wrm^{1,1}$-function will in general only be locally Lipschitz in $\Omega$. 

We are now ready for the main purpose of this section.

\begin{proof}[Proof of Theorem \ref{thm:Lagrange_W11}]
Using a standard extension theorem (see~ \cite{Mazy11SS}) we may assume that $u \in \Wrm^{1,1}(\R^d ; \R^m )$ and let for each $h \geq 0$,
\begin{equation} \label{eq:mod_cont}
  \omega(h) := \sup_{\abs{y} \leq h} \int_{\R^d} \absb{\nabla u(x+y) - \nabla u(x)} \dd x
\end{equation}
be an $\Lrm^1$-modulus of continuity for $\nabla u$. We only show the integral bound for derivatives, the corresponding bound for $u-v$ is similar, but easier. 

Apply Lemma~\ref{lem:Lagrange_est} to the regularized mapping $f = \varphi_{\eps} \conv u$, where $\phi \in \Crm_c^\infty(\R^d)$ is non-negative, has compact support 
in $B(0,1)$, integral equal to $1$, and $\phi_\eps(x) := \eps^{-d}\phi(x/\eps)$ for $\eps > 0$ to be chosen later. For integers $1 \leq k,l \leq d$ we have 
\[
  \partial_{kl}f = \frac{1}{\eps}(\partial_{l}\varphi )_{\eps} \,\conv\, (\partial_{k}u - Z^k)
  \qquad\text{for any $Z = (Z^1,\ldots,Z^d) \in \R^{m \times d}$.}
\]
Let $\tau$ be a fixed simplex in the triangulation $\Tcal$. Consequently, we may estimate using Lemma~\ref{lem:Lagrange_est},
\begin{align*}
&\int_\tau \abs{\nabla f(x)-\nabla a(x)} \dd x \\
&\qquad \leq \frac{C}{\eps}\sum_{j=0}^{d} (\diam \tau)^d \int_\tau \int_{B(0,\eps)} \frac{\abs{(\nabla \varphi)_{\eps}(y)} \cdot \abs{\nabla u(x-y) - Z}}{\abs{x-v_{j}}^{d-1}} \dd y \dd x\\
&\qquad \leq \frac{C}{\eps}\sum_{j=0}^{d} (\diam \tau)^d \eps^{-d} \norm{\nabla \varphi}_{\Lrm^{\infty}} \int_\tau \int_{B(0,\eps)} \frac{\abs{\nabla u(x-y) - Z}}{\abs{x-v_{j}}^{d-1}} \dd y \dd x\\
&\qquad \leq \frac{C}{\eps^{d+1}}\norm{\nabla \varphi}_{\Lrm^{\infty}} (\diam \tau)^d \sum_{j=0}^{d} \int_{\tau+B(0,\eps )} \abs{\nabla u(w) - Z} \int_{B(0,\eps)} \frac{\dd y}{\abs{w+y-v_{j}}^{d-1}} \dd w.
\end{align*}
Here we remark that all constants can be chosen independently of $\tau$, because the triangulation is regular. Taking
\[
  \eps = k := \sup_{\tau \in \Tcal} \diam \tau ,
\]
we note
\[
\int_{B(0,\eps)} \frac{\dd y}{\abs{w+y-v_{j}}^{d-1}}  \leq \int_{B(v_{j}-w,3\eps)} \frac{\dd y}{\abs{w+y-v_{j}}^{d-1}} = \int_{B(0,3\eps)} \frac{\dd x}{|x|^{d-1}} = 3\omega_{d} \eps,
\]
and so, for all $Z \in \R^{m \times d}$,
\begin{equation} \label{eq:interp_dist_raw_est}
\int_\tau \abs{\nabla f-\nabla a} \dd x \leq c \int_{\tau+B(0,k)} \abs{\nabla u - Z} \dd x,
\end{equation}
where $c=3(d+1)\omega_{d}C \| \varphi \|_{\Lrm^\infty}$. Let us denote $\tau^+ := \tau + B(0,k)$ and plug $Z := (\nabla u)_{\tau^+} := \dashint_{\tau^+} \nabla u \dd x$ 
into~\eqref{eq:interp_dist_raw_est}. Then we estimate
\begin{align*}
  \int_\tau \abs{\nabla f-\nabla a} \dd x &\leq c \int_{\tau^+} \abs{\nabla u - (\nabla u)_{\tau^+}} \dd x \\
  &\leq c \dashint_{\tau^+} \int_{\tau^+} \abs{\nabla u(x) - \nabla u(y)} \dd y \dd x \\
  &\leq c \dashint_{\tau^+} \int_{B(0,\diam \tau^+)} \abs{\nabla u(x) - \nabla u(x+w)} \dd w \dd x \\
  &\leq c \dashint_{B(0,\diam \tau^+)} \int_{\tau^+} \abs{\nabla u(x) - \nabla u(x+w)} \dd x \dd w. \\
  &\leq \frac{c}{k^d} \int_{B(0,3k)} \int_{\tau^+} \abs{\nabla u(x) - \nabla u(x+w)} \dd x \dd w.
\end{align*}
Next we sum these integrals over all simplices $\tau \in \Tcal$ with $\tau \cap \Omega \neq \emptyset$ and use the bounded overlap property 
of $\{ \tau^+ \}_{\tau \in \Tcal}$, which follows from the uniformity of the triangulation, to get

%
%
\begin{align*}
  \int_\Omega \abs{\nabla f-\nabla a} \dd x &\leq \frac{c_1}{k^d} \int_{B(0,3k)} \int_{\R^d} \abs{\nabla u(x) - \nabla u(x+w)} \dd x \dd w \\
  &\leq c_1 3^d \omega(3k).
\end{align*}
Here the constant $c_1$ depends on $\gamma$, $\alpha$ and the dimensions.
Using finally the standard fact that also $\norm{\nabla u - \nabla f}_{\Lrm^1(\Omega ;\R^m)} \leq \omega(k)$ we finish the proof.
\end{proof}

Remark that by virtue of the Kolmogorov-Riesz characterization of compactness in $\Lrm^{1}$, see for example~\cite{HOH}, and Theorem \ref{thm:Lagrange_W11} 
we therefore obtain a uniform rate of convergence by quasi-interpolants for all functions $u$ in any fixed compact subset of $\Wrm^{1,1}(\Omega;\R^m)$. 
The relation between compactness and regularity is also discussed in \cite{FarKri12CRCV}.

\section{Proof of Theorem \ref{thm:approx}}

We first establish the following gluing lemma:

\begin{lemma}\label{lem:element}
Let $Q$ be an open (possibly rotated) cube in $\R^d$, let $Q_{0} \subset\subset Q$ be a concentric similar open subcube, and $u \in \BV (Q ; \R^m )$ with
$\abs{Du}(\partial Q_{0})=0$. We further assume:
\begin{itemize}
\item[(i)] The closure $\cl{Q_0}$ is the disjoint union of the closure of a finite number of similar (open) rectangles 
$S_1, \ldots, S_m$ that are translations of one another along a single axis parallel to one of the sides of $Q_0$, and 
with the length $\eta$ of the short side of the $S_j$, it holds that $\dist(\partial Q, Q_0) = \eta$ (like in Figure~\ref{fig:construction}).
\item[(ii)] We are given a function $w \in \Wrm^{1,\infty}(Q_{0}; \R^m )$ that is affine when restricted to any $S_j$, and satisfies
\[
  \qquad \sum_{k=1}^{n} \int_{\partial S_{k}} \abs{u-w} \dd \Hcal^{d-1} \leq \eps \abs{Du} (Q).
\]
\end{itemize}
Then, on $A := Q\setminus \cl{Q}_{0}$ there exists a countably piecewise affine function 
$a \in \Wrm^{1,1}(A ; \R^m )$ satisfying $a=u$ on $\partial Q$, $a=w$ on $\partial Q_{0}$, and
\begin{equation}\label{paff3}
  \norm{u-a}_{\Lrm^{1}(A;\R^{m})} \leq C\eps \norm{u}_{\Lrm^{1}(A;\R^{m})},
  \quad
  \abs{\lrangle{Da}(A) - \lrangle{Du}(A)} \leq  C \eps \lrangle{Du}(Q).
\end{equation}
Here, $C = C(d,m)$ is a dimensional constant.
\end{lemma}

\begin{figure}[tb]
\def\svgwidth{\textwidth}


\begingroup%
  \makeatletter%
  \providecommand\color[2][]{%
    \errmessage{(Inkscape) Color is used for the text in Inkscape, but the package 'color.sty' is not loaded}%
    \renewcommand\color[2][]{}%
  }%
  \providecommand\transparent[1]{%
    \errmessage{(Inkscape) Transparency is used (non-zero) for the text in Inkscape, but the package 'transparent.sty' is not loaded}%
    \renewcommand\transparent[1]{}%
  }%
  \providecommand\rotatebox[2]{#2}%
  \ifx\svgwidth\undefined%
    \setlength{\unitlength}{412bp}%
    \ifx\svgscale\undefined%
      \relax%
    \else%
      \setlength{\unitlength}{\unitlength * \real{\svgscale}}%
    \fi%
  \else%
    \setlength{\unitlength}{\svgwidth}%
  \fi%
  \global\let\svgwidth\undefined%
  \global\let\svgscale\undefined%
  \makeatother%
  \begin{picture}(1,0.80582524)%
    \put(0,0){\includegraphics[width=\unitlength]{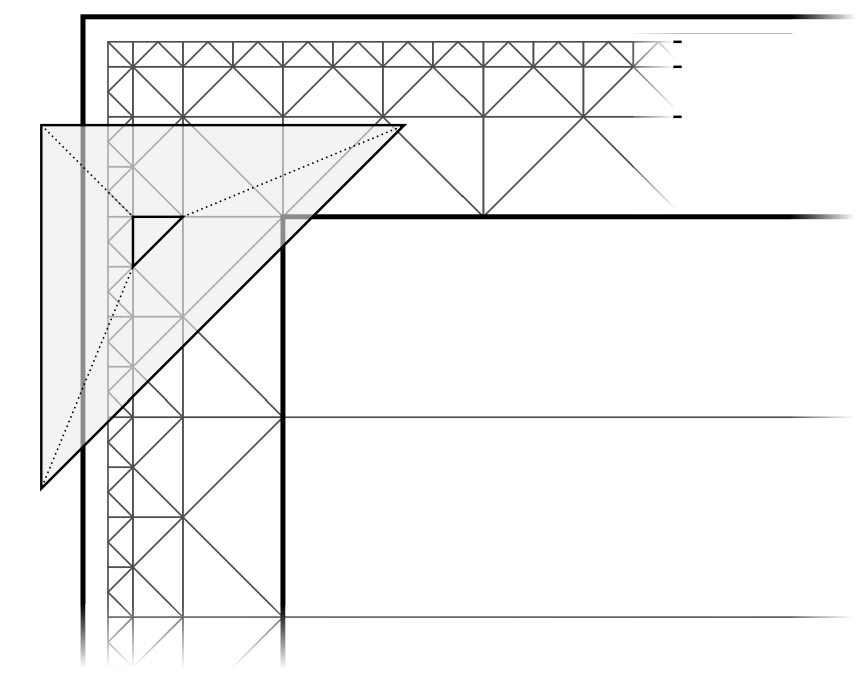}}%
    \put(0.80582524,0.74951455){\color[rgb]{0,0,0}\makebox(0,0)[lb]{\smash{$(1/2+1/4 + 1/8)\eta$}}}%
    \put(0.80582524,0.72038835){\color[rgb]{0,0,0}\makebox(0,0)[lb]{\smash{$(1/2+1/4)\eta$}}}%
    \put(0.80582524,0.66213592){\color[rgb]{0,0,0}\makebox(0,0)[lb]{\smash{$\eta/2$}}}%
    \put(0.87087379,0.57378641){\color[rgb]{0,0,0}\makebox(0,0)[lb]{\smash{$Q$}}}%
    \put(0.87087379,0.51359223){\color[rgb]{0,0,0}\makebox(0,0)[lb]{\smash{$Q_0$}}}%
    \put(0.16601942,0.52718447){\color[rgb]{0,0,0}\makebox(0,0)[lb]{\smash{$\tau$}}}%
    \put(0.00097087,0.23786408){\color[rgb]{0,0,0}\makebox(0,0)[lb]{\smash{$\tau^{(r)}$}}}%
    \put(0.34854369,0.42815534){\color[rgb]{0,0,0}\makebox(0,0)[lb]{\smash{$\eta$}}}%
    \put(0.34854369,0.19902913){\color[rgb]{0,0,0}\makebox(0,0)[lb]{\smash{$\eta$}}}%
  \end{picture}%
\endgroup%

\vspace{-10mm}
\caption{Triangulation used in the proof of Lemma~\ref{lem:element} ($\eta = \dist(Q^c,Q_0)$).}
\label{fig:construction}
\end{figure}
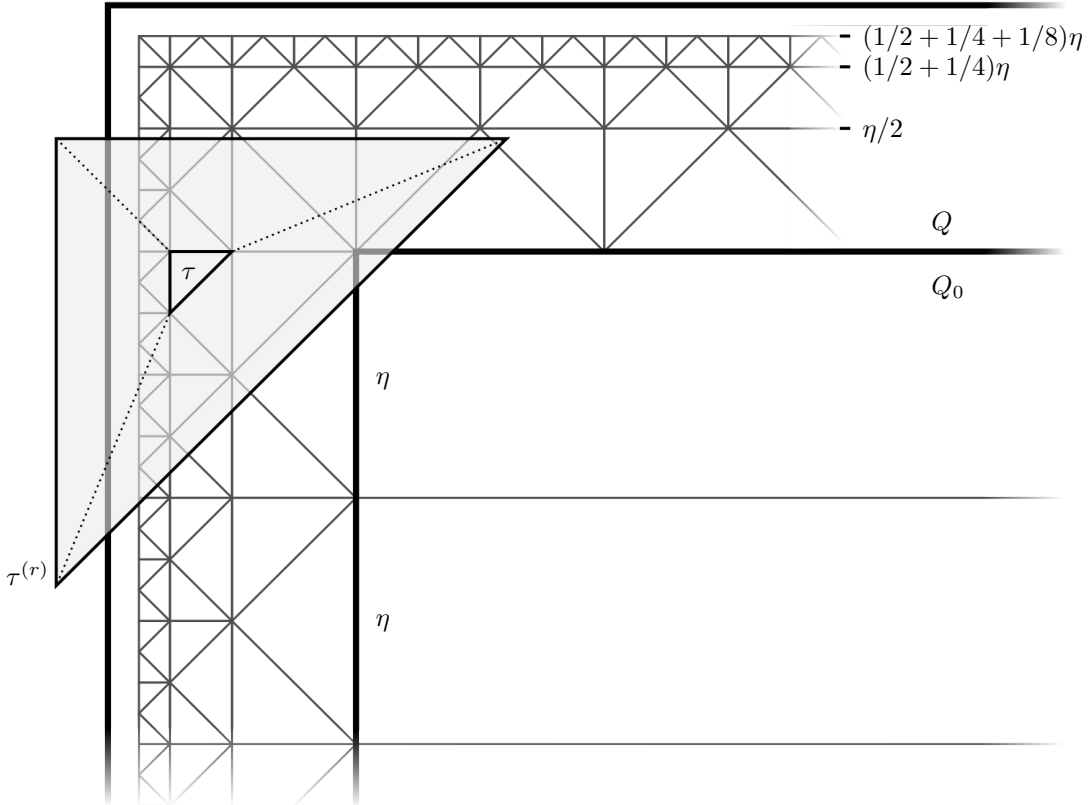

\begin{proof}[Proof of Lemma~\ref{lem:element}]
We start by selecting a smooth cut-off function $\rho \colon \overline{Q} \to [0,1]$ with $\rho = 1$ on $Q_0$,
$\rho =0$ on $\partial Q$ and $\abs{\nabla \rho} \leq 2/\mathrm{dist}(Q_{0},\partial Q)$ satisfying
\[
\int_{A} \rho \bigl(\abs{w}+\abs{u}\bigr) \dd x \leq \eps \abs{Du}(Q)
\]
and
\[
\int_{A} \bigl( \rho\abs{\nabla w -\nabla u}+\abs{w-u}\abs{\nabla \rho} \bigr) \, \dd x + \int_{A} \rho \dd \abs{D^{s}u} \leq \eps \abs{Du}(Q).
\]
We omit the routine details of this and only remark that it can be achieved for instance by suitably
mollifying the indicator function for $Q_0$ and using our assumptions. Put $g=\rho w + (1-\rho )u$, and apply 
Lemma \ref{lem:smooth_approx} to find $f \in (\Wrm^{1,1} \cap \Crm^{\infty} )(A; \R^m )$ satisfying $f=w$ on $\partial Q_0$, $f=u$ on $\partial Q$ and
\[
\int_{A} \abs{f-g} \, \dd x + \abs{\lrangle{Dg}(A)-\lrangle{Df}(A)} \leq \eps \abs{Du}(Q).
\]
Next, we construct a triangulation $\Tcal = \bigcup_{j \in \N} \tau_j$ of $A$ as in Figure~\ref{fig:construction}. 
This triangulation is of Whitney-type towards the outer boundary $\partial Q$,
\begin{equation}\label{proport}
  \diam \tau \sim \dist(\tau,\partial Q).
\end{equation}
Our $\Tcal$ also has the property that the simplices match the elements $S_1, \dots, S_n$ at the inner boundary $\partial Q_0$, meaning that 
each rectangle $S_{k} \cap \partial Q_{0}$ is a finite union of simplices $\tau \cap \partial Q_0$ where $\tau \in \Tcal$ (but not necessarily only one like in the figure). We take a sufficiently
fine triangulation until we reach the point where the corresponding $\Tcal$-piecewise affine function $a \colon A \to \R^m$ obtained by
Lagrange interpolating $g$ on the nodes of $\Tcal$ satisfies
\[
\int_{A} \! \bigl( \abs{g-a}+\abs{\lrangle{\nabla g} - \lrangle{\nabla a}} \bigr) \, \dd x \leq \eps \abs{Du}(Q).
\]
Of course the implied constants in (\ref{proport}) change when the triangulation is refined, but this is unimportant. The important fact to note
is that since the simplices shrink as we approach the outer boundary $\partial Q$ we hereby achieve that $a=g$ on $\partial Q$ (in the sense of trace),
and therefore $a=u$ on $\partial Q$. On the inner boundary $\partial Q_0$ we have $g=w$, and since $w|_{\partial Q_0}$ is piecewise affine
with respect to the (lower dimensional) triangulation $\{ \tau \cap Q_{0} \}_{\tau \in \Tcal}$ we also have that $a=w$ on $\partial Q_0$.
Finally, we estimate
\[
\int_{A} \! \abs{g-u} \, \dd x \leq \int_{A} \rho \bigl( \abs{w}+\abs{u}\bigr) \, \dd x \leq \eps \abs{Du}(Q),
\]
\begin{align*}
\abs{\lrangle{Dg}(A)-\lrangle{Du}(A)} &\leq \int_{A} \rho \bigl( \abs{\nabla w -\nabla u}+\abs{w-u}\abs{\nabla \rho} \bigr) \, \dd x
  + \int_{A} \rho \, \dd \abs{D^{s}u}\\
&\leq \eps \abs{Du}(Q).
\end{align*}
We conclude by use of the triangle inequality.
\end{proof}


We can now show our main result.

\begin{proof}[Proof of Theorem~\ref{thm:approx}]
We introduce the following notation: Let $C = (-1,1)^d$ and for $x_{0} \in \R^d$, $r>0$ set $Q(x_{0},r) := x_{0}+rC$. 
For every unit vector $n \in \mathbb{S}^{d-1}$ select a rotation $P \in \mathrm{SO}(d)$ with $Pe_{1}=n$, where $e_{1} = (1, 0, \dots, 0)^T \in \R^d$. 
Let $Q_n(x_{0},r) := x_{0}+rP(C)$, so that $Q_n(x_{0},r)$ is an open cube with center $x_0$, sidelength $2r$, and two faces orthogonal to $n$. 

\proofstep{Step 1.}
For $\Lcal^d$-almost every $x_{0} \in \Omega$ and $0<r<\mathrm{dist}(x_{0},\partial \Omega )/\sqrt{d}$ put
\[
  u_{r}(y) := \frac{u(x_{0}+ry)-\tilde{u}(x_{0})}{r}, \qquad y \in C,
\]
where $\tilde{u}(x_{0})$ is the value of the precise representative of $u$ at $x_0$ (for this and other properties of BV-functions see~\cite{AmFuPa00FBVF}). 
We recall that for $\Lcal^d$-almost all $x_{0} \in \Omega$, $u$ is approximately differentiable at $x_0$ and so, there is a sequence 
$r \todown 0$ (not specifically labelled) with the property that
\begin{equation}\label{eq:u0_G1}
  u_{r} \to u_0 \quad\text{$\lrangle{\frarg}$-strictly in $\BV(C;\R^m)$}
  \quad\text{with}\quad
  u_0(x) := \nabla u(x_{0})x, \; x \in C,
\end{equation}
with the approximate gradient $\nabla u(x_0)$ of $u$ at $x_0$. The fact that one can indeed choose a sequence $r = r_n \todown 0$ 
such that $u_r \to u_0$ $\lrangle{\frarg}$-strictly follows for example from Lemma~3.1 in~\cite{Rind12LSYM} about the construction 
of strictly converging blow-ups, applied to the measure $(Du,\Lcal^d)$. Since the trace operator is strictly continuous, we also have that
\begin{equation}\label{apd2}
  \int_{\partial C}  |u_{r}-u_0| \dd \Hcal^{d-1} \to 0 \quad\text{as}\quad  r \todown 0;
\end{equation}
see for instance~\cite{KriRin10RSIF}, pp.~53--54, and the references given there. 
Let us denote the set of such points $x_0 \in \Omega$ by $G_1$.

Now, for
\begin{equation} \label{eq:w_G1}
   w(x) = \tilde{u}(x_{0})+\nabla u(x_{0})(x-x_{0}),  \qquad x \in Q(x_0,r),
\end{equation}
from~\eqref{eq:u0_G1},~\eqref{apd2} we get by a change of variables that there exists $0 < r(x_{0}) < 1$ such that for $r < r(x_{0})$ we have
\begin{gather}
  \int_{Q(x_{0},r)} \abs{u-w} \dd x < \eps r \Lcal^{d}(Q(x_{0},r)),  \nonumber\\
  \absb{ \lrangle{Du}(Q(x_{0},r)) - \lrangle{Dw}(Q(x_{0},r)) } < \frac{\eps}{2} \Lcal^{d}(Q(x_{0},r)), \nonumber\\ 
  \int_{\partial Q(x_{0},r)}  \abs{u-w} \dd \Hcal^{d-1} < \frac{\eps}{2} \Lcal^{d}(Q(x_{0},r)). \label{apd5}
\end{gather}

\proofstep{Step 2.}
Next, for $|D^{s}u|$-almost all $x_{0} \in \Omega$, henceforth fixed, we have $\alpha_{r} := \tfrac{\abs{Du}(Q_n(x_{0},r))}{r^{d}} \to \infty$ as $r \todown 0$, and, defining
\[
  u_r(y) := \frac{u(x_{0}+ry)-u_{x_{0},r}}{r\alpha_{r}}, \qquad y \in Q_n(0,1),
\]
where $u_{x_{0},r} = \dashint_{Q_n(x_{0},r)}  u \dd x$, by Lemma~3.1 of~\cite{Rind12LSYM} we can find a sequence of $r$'s going to $0$ (not specifically labelled) such that
\begin{equation}\label{eq:ur_u0}
  u_r \to u_{0}  \quad\text{strictly in $\BV(Q_n(0,2);\R^m)$}
\end{equation}
and
\begin{equation} \label{eq:psi}
  u_{0}(y) = b\psi (y \cdot n),   \qquad y \in Q_n(0,2),
\end{equation}
where $\psi \colon (-1,1) \to \R$ is increasing and bounded. The fact that the blow-up limit $u_0$ can indeed be chosen one-directional can 
be proved via Alberti's Rank One Theorem~\cite{Albe93ROPD}, see Theorem~3.95 in~\cite{AmFuPa00FBVF}. Alternatively, we can employ the rigidity result in 
Lemma~3.2 of~\cite{Rind12LSYM} (also see Remark~5.8 in \textit{loc.\ cit.}); then we need to treat the additional case of an affine blow-up, but 
this is in fact easier than the one-directional case.

It is not hard to see that we additionally may assume that the sequence 
of $r$'s is chosen such that $\abs{Du}(\partial Q_n(x_{0},r))=0$; we again refer to~\cite{KriRin10RSIF}, pp.~54--55, for these assertions. 
Denote the set of such points $x_{0} \in \Omega$ by $G_2$, and observe that $(\Lcal^d + \abs{Du})(\Omega \setminus (G_{1} \cup G_{2}))=0$.

%

Let $N$ satisfy
\[
  N > \frac{\max\{ 2^{d+1}, (d-1)2^d \} \, \abs{b}}{\eps} (\psi (2)-\psi (-2)).
\]

For the function $\psi$ appearing in~\eqref{eq:psi} we claim that we can require that the equidistant partition
\[
  -1=t_{0}<t_{1} < \, \dots \, < t_N=1
  \qquad\text{with}\qquad t_{j+1} - t_j = \frac{2}{N}.
\]
consists only of continuity points of $\psi$. This can be achieved as follows: Select $0 < \theta < 1$ such that for the modified function 
$\psi_\theta(t) := \psi(t + \theta)$ all the $t_j$ are continuity points. Since $\psi$ has only countably many discontinuity points, such 
$\theta$ always exists. This corresponds to a blow-up sequence of the form
\[
  u_r(y) := \frac{u(x_{0} + r\theta n + ry)-u_{x_{0},r}}{r\alpha_{r}}, \qquad y \in Q_n(0,1),
\]
and in the following we need to replace $Q_n(x_0,r)$ by $Q_n(x_0 + r \theta n, r)$.
We note that $N$ still satisfies 
\[
  N > \frac{\max\{ 2^{d+1}, (d-1)2^d \} \, \abs{b}}{\eps} (\psi (1)-\psi (-1)).
\]
In the following we will however suppress such a possible shift for ease of notation.

Define $\phi$ as the piecewise affine function satisfying $\phi (t_{j})=\psi (t_{j})$ for each $j = 0,\ldots,N$,  and note that with (the rotation $P$ as above),
\[
  S_j := P[(t_{j},t_{j+1}) \times (-1,1)^{d-1}],  \qquad  j = 0,\ldots,N-1,
\]
we have
\begin{align}
  \int_{Q_n} |u_{0}(y)-b\phi (y \cdot n)| \dd y  &= 2^{d-1} \abs{b} \int_{-1}^1 \abs{\psi-\phi} \dd s  \nonumber \\
  &\leq \frac{2^d \abs{b}}{N} \sum_{j=0}^{N-1} \psi(t_{j+1}) - \psi(t_j)
    < \frac{\eps}{2} \label{eq:phi_psi_close}
\end{align}
and
\begin{align}
  &\sum_{j=0}^{N-1} \int_{\partial S_j}  |u_{0}(y)-b\phi (y \cdot n)| \dd \Hcal^{d-1}(y)  \nonumber\\
  &\qquad = |b| \sum_{j=0}^{N-1} \int_{\partial S_j}  |\psi (y \cdot n)-\phi  (y \cdot n)| \dd \Hcal^{d-1}(y)  \nonumber\\
  &\qquad \leq \frac{(d-1)2^{d-1} \abs{b}}{N} \, \sum_{j=0}^{N-1} \psi(t_{j+1}) - \psi(t_j) < \frac{\eps}{2}.  \label{eq:u0_jumps}
\end{align}
Next, in view of~\eqref{eq:ur_u0} and our choice of partition points $t_j$, 
we infer from the trace theorem that
\begin{equation}\label{apd4}
  \sum_{j=0}^{N-1} \int_{\partial S_j}  |u_r-u_{0}| \dd \Hcal^{d-1} \to 0 \mbox{ as } r \todown 0.
\end{equation}

For a point $x_0 \in G_2$, the mapping $x \mapsto b\phi (x \cdot n)$ defined above is piecewise affine. Split the the rotated cube 
$Q_n(x_{0},r)$ into $N$ rectangles $S_j(x_0,r) := x_0 + rS_j$ and define the corresponding piecewise affine map
\begin{equation} \label{eq:w_G2}
  w(x) := u_{x_0,r} + r \alpha_{r} b \, \phi \Bigl( \frac{x-x_0}{r} \cdot n \Bigr),
  \qquad \text{$x \in Q_n(x_0,r)$.}
\end{equation}
Hence, changing variables in~\eqref{eq:ur_u0},~\eqref{apd4} and using~\eqref{eq:phi_psi_close},~\eqref{eq:u0_jumps} as well as 
$\abs{Du_0}(Q_n(0,1)) = \abs{D[b\phi(y \cdot n)]}(Q_n(0,1))$ and the estimate $\abs{\sqrt{1+\abs{\frarg}^2} - \abs{\frarg}} \leq 1$, it 
follows that there exists $r(x_{0})>0$ such that for $r < r(x_{0})$,
\begin{gather}
  \int_{Q_n(x_{0},r)} \abs{u-w} \dd x < \eps r \abs{Du}(Q_n(x_{0},r)), \nonumber\\ 
  \absb{ \lrangle{Du}(Q_n(x_0,r)) - \lrangle{Dw}(Q_n(x_0,r)) } < \frac{\eps}{2} \lrangle{Du}(Q_n(x_{0},r)) + 2 \Lcal(Q_n(x_0,r)), \nonumber\\
  \sum_{j=0}^{N-1} \int_{\partial S_j(x_{0},r)}  \abs{u-w} \dd \Hcal^{d-1} < \frac{\eps}{2} \abs{Du}(Q_n(x_{0},r)). \label{apd6}
\end{gather}

\proofstep{Step 3.}
For every $x_0 \in G := G_1 \cup G_2$ we have so far constructed a (rotated) cube $Q_0 = Q_n(x_0,r)$ with $n$ and $r$ depending on $x_0$, 
and for all $x_0 \in G_2$ this cube is further subdivided into rectangles $S_j(x_{0},r)$ ($j = 0,\ldots,N-1$). Now, for every such $Q_0$ we choose a slightly 
larger concentric similar cube $Q = Q(x_0,r) \supset\supset Q_0$ with the properties
\begin{gather}
  \norm{u}_{\Lrm^1(Q \setminus \cl{Q_0};\R^m)} \leq \eps \norm{u}_{\Lrm^1(Q_0;\R^m)},  \notag\\
  \lambda (Q \setminus \cl{Q}_{0}) +\lrangle{Du}(Q \setminus \cl{Q}_{0}) + \int_{\partial Q_0} \abs{u-w} \dd \Hcal^{d-1} < \eps \lrangle{Du}(Q_0).   \label{eq:QQ0}
\end{gather}

We invoke Lemma~\ref{lem:element} with $Q_0$ and $Q$ and with $w$ as in~\eqref{eq:w_G1} or~\eqref{eq:w_G2} for 
$x_0 \in G_1$ or $x_0 \in G_2$, respectively. In particular, this yields a (countably) piecewise affine function $v_Q$ with the properties 
stated in that lemma.

\proofstep{Step 4.}
We next apply the Morse Covering Theorem~\cite{Mors47PB} (or see Theorem~5.51 in~\cite{AmFuPa00FBVF}) to cover $(\Lcal^d + \abs{Du})$-almost 
all of $\Omega$ with  (rotated) cubes $Q$ from the above family.
After subdividing the cubes $Q_0 = Q_0(x_0,r)$ corresponding to points $x_0 \in G_2$ into the rectangles $S_j(x_0,r)$ and the set 
$Q(x_0,r) \setminus Q_0(x_0,r)$ into the simplices constructed in 
Lemma~\ref{lem:element}, we thereby find a family $\Rcal$ satisfying (i) in the statement of Theorem~\ref{thm:approx}.
 
For the remaining properties (ii)--(v), we write
\[
  v = \sum_{R \in \Rcal} a_{R}\ONE_{R},
\]
where $a_R$ denotes the affine map corresponding to the rectangle or simplex $R \in \Rcal$ (in particular $v_Q = \sum_{R \subset Q} a_R$ for any $v_Q$ from 
the previous step). Because the rectangles/simplices are non-overlapping, the map $v$ is well-defined and it is clear that (ii) is satisfied. 
For the remaining assertions we consider for $j \in \N$ the mapping
\[
  v_{j} = u\ONE_{H_{j}}+\sum_{\Lcal^{d}(R) > \tfrac{1}{j}} a_{R} \ONE_{R},
  \qquad\text{with}\qquad
  H_{j} = \Omega \setminus \left( \bigcup_{\Lcal^{d}(R) > \tfrac{1}{j}}\cl{R} \right).
\]
Since the above sum is finite, we infer that $v_j \in \BV(\Omega ; \R^m )$ and
\begin{align*}
  Dv_{j} &= Du \restrict H_{j} + u \otimes \nu_{H_{j}} \, \Hcal^{d-1} \restrict \partial H_{j}\\
    &\qquad +\sum_{\Lcal^{d}(R) > \tfrac{1}{j}}\left( \nabla a_{R} \, \Lcal^{d} \restrict R + 
  a_{R} \otimes \nu_{R} \, \Hcal^{d-1} \restrict \partial R \right)\\
  &= Du \restrict H_{j} + \sum_{\Lcal^{d}(R) > \tfrac{1}{j}}\left( \nabla a_{R} \, \Lcal^{d} \restrict R
  + (a_{R}-u) \otimes \nu_{R} \, \Hcal^{d-1} \restrict \partial R \right).
\end{align*}
Here, $\nu_{H_{j}}$ and $\nu_{R}$ are the unit inner normals to $H_j$ and $R$, respectively. Employing~\eqref{apd5},~\eqref{apd6},
\begin{align*}
  |Dv_{j}|(\Omega) &= \abs{Du}(H_{j})+\sum_{\Lcal^{d}(R) > \tfrac{1}{j}}\left( |\nabla a_{R}| \, \Lcal^{d}(R)
    + \int_{\partial R}  |a_{R}-u| \dd \Hcal^{d-1}\right)\\
  &= \abs{Du}(H_{j}) + \sum_{\Lcal^{d}(R) > \tfrac{1}{j}} |\nabla a_{R}| \, \Lcal^{d}(R)+ \mathrm{O}(\eps)(\Lcal^d + \abs{Du})(\Omega).
\end{align*} 
Since $\abs{Du}(H_{j}) \to 0$ as $j \to \infty$, we see that $v \in \BV(\Omega ; \R^m )$.

Concerning (iii), we estimate using~\eqref{apd5},~\eqref{apd6},~\eqref{paff3},~\eqref{eq:QQ0},
\begin{align*}
  \absb{\lrangle{Du}(\Omega) - \lrangle{Dv}(\Omega)} &\leq \sum_Q \Bigl[ \absb{\lrangle{Du}(Q_0) - \lrangle{Dv}(Q_0)} + (\lrangle{Du} + \lrangle{Dv})(Q \setminus Q_0) \Bigr] \\
  &\leq \sum_Q \Bigl[ \eps (\Lcal^d+\lrangle{Du})(Q_0) + C(\Lcal^d + \lrangle{Du})(Q \setminus Q_0) \Bigr] + \Lcal^d(Z) \\
  &\leq \eps C \sum_Q (\Lcal^d+\lrangle{Du})(Q_0) + \Lcal^d(Z) \\
  &\leq \eps C (\Lcal^d+\lrangle{Du})(\Omega) + \Lcal^d(Z),
\end{align*}
where the summation is over all cubes used to cover $\Omega$; inside the sum $Q_0 = Q_n(x_0,r)$ refers to the \emph{inner} cube. 
The term $\Lcal^d(Z)$ originates from the additional Lebesgue measure for the cubes corresponding to singular points $x \in G_2$. 
Because they only occur on a $\Lcal^d$-negligible set, we can make this term disappear in the limit $\eps \to 0$.

By a similar calculation, also (iv) holds. We let $\Rcal_g$ denote the collection
of rectangles and simplices outside the annuli $Q\setminus Q_0$ and $\Rcal_b := \Rcal \setminus \Rcal_g$ and then we estimate:
\begin{align*}
  &\sum_{R \in \Rcal_g} \int_{\partial R} \abs{u-v} \dd \Hcal^{d-1} \\
  &\qquad = \sum_Q \left[ \sum_{j=0}^{N-1} \int_{\partial S_j(x_{0},r)}  \abs{u-v} \dd \Hcal^{d-1} + \sum_{\tau \in \Tcal} \int_{\partial \tau} \abs{u-v} \dd \Hcal^{d-1} \right] \\
  &\qquad \leq \sum_Q \Bigl[ \eps (\Lcal^d + \abs{Du})(Q) + C \abs{Du}(Q_n(x_0,r) \setminus Q_0(x_0,r)) \Bigr] \\
  &\qquad < \eps (1+C) (\Lcal^d + \abs{Du})(\Omega).
\end{align*}
For (v) we only need to observe that $v_j = u$ on $\partial \Omega$ for all $j \in \N$ and $v_j \to u$ strictly, hence the trace is preserved.

The only remaining part to check is whether the constructed mapping $v$ is of class $\Wrm^{1,1}(\Omega ;\R^m )$. The features we shall use here are 
that for every cube $Q$ as before, $v|_{Q} \in \Wrm^{1,1}(Q;\R^m )$, $v=u$ on $\partial Q$, and that we may assume that $\abs{Du}(\partial Q) = 0$.
We have
\[
  Dv = \sum_Q \bigl( Dv \restrict Q + v|_{\partial Q} \otimes \nu_{Q} \, \Hcal^{d-1} \restrict \partial Q \bigr).
\]
By assumption $v|_{\partial Q}=u|_{\partial Q}$ and 
the latter coincides also with the \emph{outer} trace of $u$ on $\partial Q$ since $\abs{Du}(\partial Q)=0$. 
Keeping in mind that $\sum_{Q \in \Qcal} \ONE_{Q} = \ONE_{\Omega}$ $\Lcal^d$-almost everywhere, and hence in the sense of $\Lrm^{1}(\R^d)$, we find
\[
  \sum_{Q \in \Qcal} v|_{\partial Q} \otimes \nu_{Q} \, \Hcal^{d-1} \restrict \partial Q
  = \sum_{Q \in \Qcal} u|_{\partial Q} \otimes D(\ONE_{Q})
  = u \otimes \nu_{\Omega} \, \Hcal^{d-1}\restrict \partial \Omega .
\]
This concludes the proof.
\end{proof}

\providecommand{\bysame}{\leavevmode\hbox to3em{\hrulefill}\thinspace}
\providecommand{\MR}{\relax\ifhmode\unskip\space\fi MR }
\providecommand{\MRhref}[2]{%
  \href{http://www.ams.org/mathscinet-getitem?mr=#1}{#2}
}
\providecommand{\href}[2]{#2}

\end{document}